\documentclass[oneside,english]{amsart}
\usepackage[T1]{fontenc}
\usepackage[latin9]{inputenc}
\usepackage{geometry}
\usepackage{amstext}
\usepackage{amsthm}
\usepackage{amssymb}
\usepackage[all]{xy}
\PassOptionsToPackage{normalem}{ulem}
\usepackage{ulem}

\makeatletter

\newcommand{\noun}[1]{\textsc{#1}}

\numberwithin{equation}{section}
\numberwithin{figure}{section}
\theoremstyle{plain}
\newtheorem{thm}{\protect\theoremname}
  \theoremstyle{definition}
  \newtheorem{defn}[thm]{\protect\definitionname}
  \theoremstyle{remark}
  \newtheorem*{rem*}{\protect\remarkname}
  \theoremstyle{plain}
  \newtheorem{prop}[thm]{\protect\propositionname}
  \theoremstyle{definition}
  \newtheorem*{example*}{\protect\examplename}
  \theoremstyle{definition}
  \newtheorem{example}[thm]{\protect\examplename}

\newcommand{\xyR}[1]{
  \xydef@\xymatrixrowsep@{#1}}
\newcommand{\xyC}[1]{
  \xydef@\xymatrixcolsep@{#1}}

\makeatother
\usepackage{hyperref}
\hypersetup{
    colorlinks=true,       
    linkcolor=blue,          
    citecolor=blue,        
    filecolor=blue,      
    urlcolor=blue           
}

\usepackage{babel}
  \providecommand{\definitionname}{Definition}
  \providecommand{\examplename}{Example}
  \providecommand{\propositionname}{Proposition}
  \providecommand{\remarkname}{Remark}
\providecommand{\theoremname}{Theorem}

\begin{document}

\title{Smooth contractible threefolds with hyperbolic $\mathbb{G}_{m}$-actions
via ps-divisors}

\author{Charlie Petitjean}

\email{petitjean.charlie@gmail.com}
\begin{abstract}
The aim of this note is to give an alternative proof of the theorem
4.1 of \cite{KR4}, that is, a characterization of smooth contractible
affine varieties endowed with a hyperbolic action of the group $\mathbb{G}_{m}\simeq\mathbb{C}^{\text{*}}$,
using the language of polyhedral divisors developed in \cite{A-H}
as generalization of $\mathbb{Q}$-divisors.
\end{abstract}

\address{Instituto de Matem\'atica y F\'isica, Universidad de Talca, Casilla
721, Talca, Chile}



\thanks{This research was partially supported by project Fondecyt postdoctorado
3160005.}

\maketitle
In \cite{KR4}, Koras and Russell provided a characterization of smooth
contractible affine varieties endowed with a hyperbolic action of
the group $\mathbb{G}_{m}\simeq\mathbb{C}^{\text{*}}$. This is an
important step in the proof of the linearization conjecture in dimension
three \cite{KaKMLR}. The conjecture states that every $\mathbb{G}_{m}$-action
on the affine $n$-space is linearizable, that is, up to a conjugation
by an automorphism of $\mathbb{A}^{n}$, we can assume that the action
is linear. The case $n=1$ is trivial and corresponds to the simplest
toric case, $\mathbb{G}_{m}$ acting on $\mathbb{A}^{1}$, for $n=2$
the original proof is due to Gutwirth \cite{G}, the case $n=3$ is
more difficult and obtained after a long series of articles initiated
by Kambayashi and Russell \cite{KamR} continued with Koras \cite{KR1,KR2,KR3,KR4}
and achieved with the contribution of Kaliman and Makar-Limanov \cite{KaML,KaKMLR}.
In higher dimension, that is for $n>3$, no general results are known.

On the other hand, the case of $\mathbb{G}_{m}$-action on $\mathbb{A}^{3}$
can be also viewed as the first case of a linearization conjecture
of complexity two, that is, algebraic torus actions of dimension $n-2$
acting effectively on $\mathbb{A}^{n}$. The linearization conjecture
of complexity zero corresponds to the toric case and it is true, as
the linearization conjecture of complexity one by a result of Bia\l ynicki-Birula
\cite{BB}. Here, once again, in higher dimension no general results
are known.

The purpose is to use a geometrico-combinatorial presentation of normal
varieties endowed with an effective algebraic torus action, developed
by Altmann and Hausen \cite{A-H} several years after the result of
Koras and Russell. This presentation is a generalization of two presentations,
the first one is the presentation of affine toric varieties via cones
in lattices \cite{CLS} and the second is the presentation of $\mathbb{G}_{m}$-surfaces
via $\mathbb{Q}$-divisors \cite{De,FZ}. A recollection in our particular
case will be made in the section one. The presentation of a smooth
affine variety $X$ endowed with a hyperbolic action of $\mathbb{G}_{m}$
consists of a pair $(Y,\mathcal{D})$ such that $Y$ is a normal variety
of dimension $\mathrm{dim}(X)-1$ playing essentially a role of quotient
and $\mathcal{D}$ is a divisor on $Y$ with particular coefficients
attached on each irreducible components. As it can be expected, this
presentation is particularly adapted to give a shorter proof of the
charactersation given initially by Theorem 4.1 in \cite{KR4}, this
is realized in section two. We hope that this note serves to clarify
and condense existing results in the field and can be the first step
for possible generalizations.

\section{Hyperbolic $\mathbb{G}_{m}$-actions on smooth affine varieties}

Let $X=\mathrm{Spec}(A)$ be an affine smooth variety endowed with
a hyperbolic $\mathbb{G}_{m}$-action. First, recall that the coordinate
ring $A$ of $X$ is $\mathbb{Z}$-graded in a natural way by its
subspaces of semi-invariants of weight $n$, $A_{n}$, for the effective
$\mathbb{G}_{m}$-action on $X$: 
\[
A_{n}:=\left\{ f\in A/f(\lambda\cdot x)=\lambda^{n}f(x),\forall\lambda\in\mathbb{G}_{m}\right\} .
\]
The action is said to be \emph{hyperbolic} if there is at least one
$n_{1}<0$ and one $n_{2}>0$ such that $A_{n_{1}}$ and $A_{n_{2}}$
are nonzero. In particular $A_{0}$ is the ring of invariant functions
on $X$, thus $q:X\rightarrow Y_{0}(X):=X/\!/\mathbb{G}_{m}=\mathrm{Spec}(A_{0})$
is the algebraic quotient to the $\mathbb{G}_{m}$-action on $X$. 

\noindent
\begin{defn}
\label{The-A-H-quotient-1} The \emph{A-H quotient} $Y(X)$ of $X$
is the blow-up $\pi:Y(X)\rightarrow Y_{0}(X)$ of $Y_{0}(X)$ with
center at the closed subscheme defined by the ideal $\mathcal{I}=\left\langle A_{d}\cdot A_{-d}\right\rangle $,
where $d>0$ is chosen such that $\bigoplus_{n\in\mathbb{Z}}A_{dn}$
is generated by $A_{0}$ and $A_{\pm d}$. 
\end{defn}
\begin{rem*}
If $X$ admits a unique fixed point $x_{0}$, then the center of the
blow-up is supported at the image of $x_{0}$ by the algebraic quotient
morphism $q$.
\end{rem*}
\begin{defn}
A \emph{segmental divisor} $\mathcal{D}$ on an algebraic variety
$Y$ is a formal finite sum $\mathcal{D}=\sum[o_{i},p_{i}]\otimes D_{i}$,
where $D_{i}$ are prime Weil divisors on $Y$ and $[o_{i},p_{i}]$
are closed intervals with rational bounds $o_{i}\leq p_{i}$. Every
element $n\in\mathbb{Z}$ determines a map from segmental divisors
to the group of Weil $\mathbb{Q}$-divisors on $Y$:
\end{defn}
\[
\mathcal{D}=\sum[o_{i},p_{i}]\otimes D_{i}\rightarrow\mathcal{D}(n)=\sum\left(\mathrm{min}[no_{i},np_{i}]\right)D_{i}=\sum q_{i}D_{i}.
\]

\begin{defn}
\label{A-porper-segmental-divisor,}A \emph{proper-segmental divisor
}(ps-divisor) $\mathcal{D}$ on a variety $Y$ is a segmental divisor
on $Y$ such that for every $n\in\mathbb{Z}$, $\mathcal{D}(n)$ satisfies
the following properties:
\end{defn}
\begin{enumerate}
\item [$(i)$] $\mathcal{D}(n)$ is a $\mathbb{Q}$-Cartier divisor on $Y$.
\item [$(ii)$] $\mathcal{D}(n)$ is semi-ample, that is, for some $p\in\mathbb{Z}_{>0}$,
$Y$ is covered by complements of supports of effective divisors linearly
equivalent to $\mathcal{D}(pn)$.
\item [$(iii)$] $\mathcal{D}(n)$ is big, that is, for some $p\in\mathbb{Z}_{>0}$,
there exists an effective divisor $D$ linearly equivalent to $\mathcal{D}(pn)$
such that $Y\setminus\mathrm{Supp}(D)$ is affine. 
\end{enumerate}
\noindent
\begin{defn}
A variety $Y$ is said to be semi-projective if its ring of regular
functions $\Gamma(Y,\mathcal{O}_{Y})$ is finitely generated and $Y$
is projective over $Y_{0}=\mathrm{Spec}(\Gamma(Y,\mathcal{O}_{Y}))$.
\end{defn}
Considering a ps-divisor $\mathcal{D}$ on a semi-projective variety
$Y$ the $\mathbb{Z}$-graded algebra

\[
A=\bigoplus_{n\in\mathbb{Z}}A_{n}=\bigoplus_{n\in\mathbb{Z}}\Gamma(Y,\mathcal{O}_{Y}(\mathcal{D}(n))),
\]

\begin{flushleft}
is finitely generated. The associated affine variety $X=\mathrm{Spec}(A)$
is therefore a $\mathbb{G}_{m}$-variety. In the case of hyperbolic
$\mathbb{G}_{m}$-action, the main theorem of \cite{A-H} can be reformulated
as follows:
\par\end{flushleft}
\begin{thm}
For any ps-divisor $\mathcal{D}$ on a normal semi-projective variety
$Y$ the scheme
\begin{flushleft}
\[
\mathbb{S}(Y,\mathcal{D})=\mathrm{Spec}(\bigoplus_{n\in\mathbb{Z}}\Gamma(Y,\mathcal{O}_{Y}(\mathcal{D}(n))))
\]
is a normal affine variety of dimension $\mathrm{dim}(Y)+1$ endowed
with an effective hyperbolic $\mathbb{G}_{m}$-action, whose A-H quotient
$Y(\mathbb{S}(Y,\mathcal{D}))$ is birationally isomorphic to $Y$.
\par\end{flushleft}
Conversely any normal affine variety $X$ endowed with an effective
hyperbolic $\mathbb{G}_{m}$-action is isomorphic to $\mathbb{S}(Y(X),\mathcal{D})$
for a suitable ps-divisor $\mathcal{D}$ on $Y(X)$.
\end{thm}
\noindent

A ps-divisor $\mathcal{D}$ such that $X\simeq\mathbb{S}(Y,\mathcal{D})$
can be obtained by the following downgrading (see \cite[section 11]{A-H})\uline{:}

Consider $X$ embedded as a $\mathbb{G}_{m}$-stable subvariety of
an affine affine toric variety. The calculation is then reduced to
the toric case by considering an embedding in $\mathbb{A}^{n}$ endowed
with a linear action of a torus $\mathbb{T}$ for a sufficiently large
$n$. By this way, the inclusion of $\mathbb{G}_{m}\hookrightarrow\mathbb{T}$
corresponds to an inclusion of the lattice $\mathbb{Z}$ of one parameter
subgroups of $\mathbb{G}_{m}$ in the lattice $\mathbb{Z}^{n}=N$
of one parameter subgroups of $\mathbb{T}$. We obtain the exact sequence:
\[
\xymatrix{0\ar[r] & \mathbb{Z}\ar[r]_{F} & N=\mathbb{Z}^{n}\ar[r]_{P}\ar@/_{1pc}/[l]_{s} & N'=\mathbb{Z}^{n}/\mathbb{Z}\ar[r] & 0}
,
\]
 where $F$ is given by the induced action of $\mathbb{G}_{m}$ on
$\mathbb{A}^{n}$ (it is the \emph{matrix of weights}) and $s$ is
a section of $F$. Let $v_{i}$, for $i=1,\ldots,n$, be the first
integral vectors of the unidimensional cone generated by the i-th
column vector of $P$ considered as rays in the lattice $N'\simeq\mathbb{Z}^{n-1}$.
Let $Z$ be the toric variety, of maximal dimension $(n-1)$, determined
by the coarsest fan containing all cones generated by subsets of $\{v_{1},\ldots,v_{n}\}$
in $N'$. Then each $v_{i}$ corresponds to a $\mathbb{T}'$-invariant
divisor where $\mathbb{T}'=\mathrm{Spec}(\mathbb{C}[N'^{\vee}]$,
for $i=1,\ldots,n$. By \cite[section 11]{A-H}, $Z$ contains the
A-H quotient $Y$ of $X$, as sub-variety, and the support of $D_{i}$
is obtained by restricting the $\mathbb{T}'$-invariant divisor corresponding
to $v_{i}$ to $Y$. The segment associated to the divisor $D_{i}$
is equal to $s(\mathbb{R}_{\geq0}^{n}\cap P^{-1}(v_{i}))$, it can
occur that the segment is a point. If $X$ is the affine space endowed
with a linear action of $\mathbb{G}_{m}$, the embedding of $X$ as
a $\mathbb{G}_{m}$-stable subvariety of an affine toric variety is
reduced to the identity. In this case, the A-H quotient of $\mathbb{A}^{n}$
is $Z$ itself.

Linear hyperbolic $\mathbb{G}_{m}$-actions on $\mathbb{A}^{3}$ have
been fully characterized by this method, see for instance \cite{L-P}.
Let $F=\overset{}{}^{t}(a_{1},a_{2},a_{3})$ where $-a_{1}$, $a_{2}$
and $a_{3}$ are strictly coprime positive integers, and let $s=(\alpha,\beta,\gamma)$
where $\alpha$, $\beta$, $\gamma$ are integers such that $s\circ F=1$.
Let $\rho(i,j)=\mathrm{gcd}(i,j)$ for $i,j=a_{1},a_{2},a_{3}$and
$\delta=\mathrm{gcd}(\frac{a_{2}}{\rho(a_{1},a_{2})},\frac{a_{3}}{\rho(a_{1},a_{3})})$.
In the case of linear hyperbolic $\mathbb{G}_{m}$-actions on $\mathbb{A}^{3}$,
\cite[Proposition 11]{L-P} gives:
\begin{prop}
\label{classification A3  } Let $X\simeq\mathbb{A}^{3}=\mathrm{Spec}(\mathbb{C}[x,z,t])$
endowed with a linear hyperbolic $\mathbb{G}_{m}$-action. Then $\mathbb{A}^{3}$
is the $\mathbb{G}_{m}$-equivariant cyclic quotient of a variety
$X'$ isomorphic to $\mathbb{A}^{3}$ and such that $X'$ is equivariantly
isomorphic to the $\mathbb{G}_{m}$-variety $\mathbb{S}(Y,\mathcal{D})$
with $Y$ and $\mathcal{D}$  defined as follows: 
\end{prop}
\begin{enumerate}
\item [$(i)$] $Y$ is isomorphic to a blow-up $\pi:\tilde{\mathbb{A}}^{2}\rightarrow\mathbb{A}^{2}$
of $\mathbb{A}^{2}$ at the origin.
\item [$(ii)$] $\mathcal{D}$ is of the form: 
\[
\mathcal{D}=\left\{ \frac{\alpha\rho(a_{1},a_{2})}{-a_{1}}\right\} \otimes D_{2}+\left\{ \frac{\beta\rho(a_{1},a_{3})}{-a_{1}}\right\} \otimes D_{3}+\left[\frac{\gamma}{\delta},\frac{\gamma}{\delta}+\frac{1}{-\delta a_{1}}\right]\otimes E,
\]
 with $D_{2}$, $D_{3}$ are the strict transforms of the coordinate
axes, $E$ is the exceptional divisor of $\pi$ and the linear $\mathbb{G}_{m}$-action
on $X$ is given by its matrix of weights $F=\overset{}{}^{t}(a_{1},a_{2},a_{3})$
with $-a_{1}$, $a_{2}$ and $a_{3}$ strictly positive integers.
\end{enumerate}
Let $X$ be equivariantly isomorphic to $\mathbb{S}(Y,\mathcal{D})$.
Then segmental prime divisors $[a_{i},b_{i}]\otimes D_{i}$ occurring
in the ps-divisor $\mathcal{D}=\sum_{i=1}^{n}[o_{i},p_{i}]\otimes D_{i}$
encode, according to their coefficients, two distinct facts (see \cite{P}
and \cite[Theorem 4.18]{FZ} ) :
\begin{enumerate}
\item [$(i)$] For any $i$ such that $[o_{i},p_{i}]$ is reduced to a rational
point $\{p_{i}\}$, the divisor $\{p_{i}\}\otimes D_{i}$ encodes
isotropy subgroup of finite order, or equivalently the fact that $X$
is the finite cyclic cover of an other $\mathbb{G}_{m}$-variety having
a A-H quotient isomorphic to that of $X$ and where $D_{i}$ does
not appear in the the presentation.
\item [$(ii)$] For any $i$ such that $[o_{i},p_{i}]$ is an interval with
non-empty interior, the divisor $[o_{i},p_{i}]\otimes D_{i}$ encodes
isotropy subgroup of infinite order, equivalently fixed points.
\end{enumerate}
In particular as we have assumed that $X$ admits a unique fixed point,
the only divisor whose coefficient is an interval with non-empty interior
is the exceptional divisor of the blow-up $\pi:Y\rightarrow Y_{0}$.\\
The smoothness of the $\mathbb{G}_{m}$-threefold $X$ can be checked
using \cite[Theorem 7 and proposition 11]{L-P}: 
\begin{enumerate}
\item [$(iii)$]If $X=\mathbb{S}(Y,\mathcal{D})$ is an affine $\mathbb{G}_{m}$-variety
of dimension $n$, then, $X$ is smooth if and only if the combinatorial
data $(Y,\mathcal{D})$ is locally isomorphic in the \'etale topology
to the combinatorial data of the affine space endowed with a linear
$\mathbb{G}_{m}$-action. 
\end{enumerate}
\begin{example*}
Let $X\simeq\mathbb{A}^{3}=\mathrm{Spec}(\mathbb{C}[x,y,z])$ endowed
with the $\mathbb{G}_{m}$-action with matrix of weights $F=\overset{}{}^{t}(2,3,-6)$.
Then $\mathbb{A}^{3}$ is equivariantly isomorphic to $\mathbb{S}(\tilde{\mathbb{A}}^{2},\mathcal{D})$
where $\pi:\tilde{\mathbb{A}}^{2}\rightarrow\mathbb{A}^{2}$ is the
blow-up of the algebraic quotient $\mathbb{A}^{2}=\mathrm{Spec}(\mathbb{C}[u,v])\mathrm{=Spec}(\mathbb{C}[x^{3}z,y^{2}z])$
with center at the closed subscheme defined by the ideal $\mathcal{I}=(u,v)$,
the p-s divisor $\mathcal{D}$ is of the form

\[
\mathcal{D}=\left\{ -\frac{1}{3}\right\} \otimes D_{1}+\left\{ \frac{1}{2}\right\} \otimes D_{2}+\left[0,\frac{1}{6}\right]\otimes E,
\]
 where $D_{1}$, $D_{2}$ are the strict transforms of the coordinates
axes $\{u=0\}$ and $\{v=0\}$, and $E$ is the exceptional divisor
of the blow-up $\pi$.

$\mathbb{A}^{3}$ admits commuting actions of the groups $\mu_{2}$
and $\mu_{3}$ of square and cubic roots of the unity defined respectively
by $\epsilon\times(x,y,z)\rightarrow(\epsilon x,y,\epsilon z)$ and
$\varsigma\times(x,y,z)\rightarrow(x,\varsigma y,\varsigma z)$. These
commute with the $\mathbb{G}_{m}$-action and $X\simeq\mathbb{A}^{3}$
has a structure of a $\mathbb{G}_{m}$-equivariant bi-cyclic cover
of $X//(\mu_{2}\times\mu_{3})\simeq\mathbb{A}^{3}$ with matrix of
weights $F=\overset{}{}^{t}(1,1,-1)$. Applying the main Theorem of
\cite{P}, $X//(\mu_{2}\times\mu_{3})$ is equivariantly isomorphic
to $\mathbb{S}(Y,[-1,0]\otimes E)$. The ring of regular functions
of $X$ is $\mathbb{Z}$-graded via $M$ the character lattice of
$\mathbb{G}_{m}$ and the ring of regular functions of $X//(\mu_{2}\times\mu_{3})$
is $\mathbb{Z}$-graded via a sublattice $M'\subset M$ of index $6$,
thus A-H presentation of $X//(\mu_{2}\times\mu_{3})$ is obtained
considering the same A-H quotient $Y$ and a multiple of $\mathcal{D}$,
namely, $6\mathcal{D}\sim[-1,0]\otimes E$. The structure of $\mathbb{G}_{m}$-equivariant
bi-cyclic covering implies that the induced action of $\mu_{2}\times\mu_{3}$
on $Y$ is trivial so that $Y\simeq Y(X//(\mu_{2}\times\mu_{3}))$
. So, only the ps-divisors change as the lattice changes (see \cite{P}).
This is illustrated by the following diagram,
\end{example*}
\begin{center}
\[
\xyC{2pc}\xymatrix{ & X\simeq\mathbb{A}^{3}\simeq\mathbb{S}(Y,\mathcal{D})\ar[ld]\ar[rd]\\
X//\mu_{3}\simeq\mathbb{A}^{3}\simeq\mathbb{S}(Y,3\mathcal{D})\ar[rd] &  & X//\mu_{2}\simeq\mathbb{A}^{3}\simeq\mathbb{S}(Y,2\mathcal{D})\ar[ld]\\
 & X//(\mu_{2}\times\mu_{3})\simeq\mathbb{A}^{3}\simeq\mathbb{S}(Y,6\mathcal{D})
}
\]
\par\end{center}

\noindent

One way to constructs many examples of varieties endowed with hyperbolic
$\mathbb{G}_{m}$-action is the following (see \cite{T,Z}):
\begin{defn}
\label{Hyperbolic modification}Let $X=\{f(u_{2},\ldots,u_{n})\}\subset\mathbb{A}^{n-1}$
be a smooth hypersurface defined by the zeros of the polynomial $f$.
\emph{The hyperbolic modification of $X$} is obtained blowing-up
the variety $\mathbb{A}^{1}\times X\subset\mathbb{A}^{1}\times\mathbb{A}^{n-1}$
with center at the origin and removing the proper transform of $\{0\}\times X$.
By this way we obtain a variety $X_{f}=\{f(x_{1}x_{2},\ldots,x_{1}x_{n})/x_{1}=0\}\subset\mathbb{A}^{n}$
which is stable for the following hyperbolic $\mathbb{G}_{m}$-action
on $\mathbb{A}^{n}$, $\lambda\cdot(x_{1},x_{2},\ldots,x_{n})=(\lambda^{-1}x_{1},\lambda x_{2},\ldots,\lambda x_{n})$.
\end{defn}
Each of these hyperbolic modifications will be characterized by the
following A-H presentation: $X_{f}$ is equivariantly isomorphic to
$\mathbb{S}(\tilde{X},[-1,0]\otimes E)$ where $\pi:\tilde{X}\rightarrow X$
is the blow-up of $X$ with center at the closed subscheme defined
by the ideal $\mathcal{I}=(u_{2},\ldots,u_{n})=(x_{1}x_{2},\ldots,x_{1}x_{n})$,
and $E$ is the exceptional divisor of the blow-up.

\section*{Constructions of koras-Russell threefolds and $\mathbb{A}^{3}$ via
segmental divisors}

In this section we will consider an approach to the classification
given by Koras and Russell using segmental divisors in an \'etale neighborhood
of a fixed point to determine all possibles configurations for the
threefolds.

The tangent space of a smooth $\mathbb{G}_{m}$-variety at the fixed
point is an affine three-space with a linear $\mathbb{G}_{m}$-action.
The action $\mathbb{G}_{m}\times\mathbb{A}^{3}\rightarrow\mathbb{A}^{3}$
is given by $\lambda\cdot(x,y,z)=(\lambda^{a_{1}}x,\lambda^{a_{2}}y,\lambda^{a_{3}}z)$
and characterized by its matrix of weight $F=\overset{}{}^{t}(a_{1},a_{2},a_{3})$.

In \cite{K}, Koras proves that if $\mathbb{A}^{3}$ is endowed with
an algebraic $\mathbb{G}_{m}$-action such that the algebraic quotient
is of dimension $2$, then the quotient is isomorphic to $\mathbb{A}^{2}/\mu$,
where $\mu$ is a finite cyclic group. If a variety $X$ has such
algebraic quotient by a $\mathbb{G}_{m}$-action it was said that
$X$ has quotient as expected for a $\mathbb{G}_{m}$-action on $\mathbb{A}^{3}$.

If $X$ is endowed with an action of the group $\mu_{i}$ of $i$-th
roots of the unity commuting with the $\mathbb{G}_{m}$-action, then
its algebraic quotient $X/\!/\mathbb{G}_{m}$ inherits also of a $\mu_{i}$-action,
possibly trivial.
\begin{thm}
\label{thm:A-smooth,-contractible}A smooth, contractible, affine
threefolds $X=\mathrm{Spec}(A)$ with a hyperbolic $\mathbb{G}_{m}$-action,
an unique fixed point and an algebraic quotient isomorphic to $\mathbb{A}^{2}/\mu$,
where $\mu$ is a finite cyclic group, is determined by the following
data:

$(a)$ A triple of weights $a_{1},a_{2},a_{3}$ with $-a_{1},a_{2},a_{3}>0$.
These define a hyperbolic $\mathbb{G}_{m}$-action on the tangent
space of the fixed point $\mathbb{A}^{3}$ via the matrix weight $F=\overset{}{}^{t}(a_{1},a_{2},a_{3})$ 

$(b)$ A triple of positive integers $s=(\alpha,\beta,\gamma)$ such
that $s\circ F=1$.

$(c)$ Curves $C_{2}$, $C_{3}$ in $\mathbb{A}^{2}$ satisfying:

\quad$(i)$ $C_{i}$ is defined by an $\mu_{a_{1}}$-homogeneous
polynomial.

\quad$(ii)$ $C_{i}\simeq\mathbb{A}^{1}$

\quad$(iii)$ $C_{2}$ and $C_{3}$ meet normally in $(0,0)$ and
$d-1$ additional points.

\end{thm}
\begin{proof}
Let $X$ be an affine three space endowed with a hyperbolic $\mathbb{G}_{m}$-action,
we assume that $X$ is equivariantly isomorphic to $\mathbb{S}(Y,\mathcal{D})$
where the pair is minimal constructed as in section one.

First step: the A-H quotient.

By assumption the fixed point set $X^{\mathbb{G}_{m}}$ of $X$ is
of dimension $0$, and by \cite{BB} it is non-empty and connected
and therefore reduced to a unique point, $x_{0}\in X$. Using the
result of \cite{K}, the algebraic quotient $Y_{0}$ of $X$ by the
$\mathbb{G}_{m}$-action is isomorphic to $\mathbb{A}^{2}/\!/\mu$
where $\mu$ is a finite cyclic group. So the A-H quotient $Y$ of
$X$ is isomorphic to a blow-up of $Y_{0}\simeq\mathbb{A}^{2}/\!/\mu$
supported on the image of $x_{0}$ by the quotient morphism $q:X\rightarrow Y_{0}\simeq\mathbb{A}^{2}/\!/\mu$.
Since $X$ is smooth, it follow from \cite{L-P} that there exists
an \'etale $\mathbb{G}_{m}$-invariant neighborhood $\mathcal{U}$ of
$x_{0}$, which is equivariantly isomorphic to an \'etale neighborhood
of a fixed point in $\mathbb{A}^{3}=\mathrm{Spec}(\mathbb{C}[x,y,z])$
endowed with a linear hyperbolic $\mathbb{G}_{m}$-action of the form,
$\lambda\cdot(x,y,z)=(\lambda^{a_{1}}x,\lambda^{a_{2}}y,\lambda^{a_{3}}z)$
with $-a_{1},a_{2},a_{3}>0$. So $Y(X)$ is isomorphic to the A-H
quotient of the previous $\mathbb{A}^{3}$ and determined by the triple
of reduced weights $a_{1},a_{2},a_{3}$.

Considering an embedding of $X$ as a \noun{$\mathbb{G}_{m}$-}stable
subvariety of an affine space endowed with a linear $\mathbb{G}_{m}$-action,
we can always construct a finite cyclic cover along coordinate axes
of the ambient space such that the new variety obtained by this finite
ramified morphism admits an algebraic quotient isomorphic to the complex
plane $\mathbb{A}^{2}$. By this way and using \cite{P}, we can assume
that the A-H presentation of $X$ is completely determined by that
of the cyclic cover and the data of the cyclic group. 

Second step: the ps-divisor $\mathcal{D}$ in an \'etale neighborhood
of the exceptional divisor.

In terms of ps-divisor as $X$ is a smooth $\mathbb{G}_{m}$-variety,
the smoothness criterion gives us that there exists $\mathcal{V}$,
in $Y$, an \'etale neighborhood of the exceptional divisor such that
$\mathbb{S}(\mathcal{V},\mathcal{D}_{\mid\mathcal{V}})$ is equivariantly
isomorphic to an \'etale neighborhood of a fixed point in $\mathbb{A}^{3}$
endowed with a linear hyperbolic $\mathbb{G}_{m}$-action.

Moreover as $X$ is smooth in the \'etale neighborhood of the fixed
point, by \cite{L-P}, $\mathcal{D}_{\mid\mathcal{V}}$ is of the
form:

\[
\mathcal{D}_{\mid\mathcal{V}}=\left\{ \frac{\alpha\rho(a_{1},a_{2})}{-a_{1}}\right\} \otimes D_{2}+\left\{ \frac{\beta\rho(a_{1},a_{3})}{-a_{1}}\right\} \otimes D_{3}+\left[\frac{\gamma}{\delta},\frac{\gamma}{\delta}+\frac{1}{-\delta a_{1}}\right]\otimes E,
\]

\begin{flushleft}
where Proposition \ref{classification A3  } provides the coefficients
and the supports of the divisors in the \'etale neighborhood.
\par\end{flushleft}
Third step: the number of irreducible component of the support of
$\mathcal{D}$.

The ps-divisor$\mathcal{D}$ is completely determined by $\mathcal{D}_{\mid\mathcal{V}}$.
Indeed as $X$ admits a unique fixed point by the first section, $E$
is the unique divisor whose coefficient is an interval with non-empty
interior. Now suppose that 
\[
\mathcal{D}=\left\{ \frac{\alpha\rho(a_{1},a_{2})}{-a_{1}}\right\} \otimes D_{2}+\left\{ \frac{\beta\rho(a_{1},a_{3})}{-a_{1}}\right\} \otimes D_{3}+\left[\frac{\gamma}{\delta},\frac{\gamma}{\delta}+\frac{1}{-\delta a_{1}}\right]\otimes E+\sum_{i=4}^{p}\left\{ q_{i}\right\} \otimes D_{i},
\]
 
\[
\mathrm{with}\,\,\mathcal{D}_{\mid\mathcal{V}}=\left\{ \frac{\alpha\rho(a,c)}{c}\right\} \otimes D_{2\mid\mathfrak{U}}+\left\{ \frac{\beta\rho(b,c)}{c}\right\} \otimes D_{3\mid\mathfrak{U}}+\left[\frac{\gamma}{\delta},\frac{\gamma}{\delta}+\frac{1}{\delta c}\right]\otimes E.
\]
 By the first section, this implies that there exists a $\mathbb{G}_{m}$-variety
$\hat{X}$ such that $c:X\rightarrow\hat{X}$ is an equivariant cyclic
cover of $\hat{X}$ along divisors whose images in the algebraic quotient
do not intersect the image of the fixed point and thus $X$ does not
admit a fixed point. This contradicts the assumption and so $X$ is
equivariantly isomorphic to $\mathbb{S}(Y,\mathcal{D})$ where $Y$
is isomorphic to the A-H quotient of an $\mathbb{A}^{3}$ endowed
with a hyperbolic linear $\mathbb{G}_{m}$-action and $\mathcal{D}$
is of the form:

\[
\mathcal{D}=\left\{ \frac{\alpha\rho(a_{1},a_{2})}{-a_{1}}\right\} \otimes D_{2}+\left\{ \frac{\beta\rho(a_{1},a_{3})}{-a_{1}}\right\} \otimes D_{3}+\left[\frac{\gamma}{\delta},\frac{\gamma}{\delta}+\frac{1}{-\delta a_{1}}\right]\otimes E,
\]
where $E$ is the exceptional divisor of $\pi:Y\rightarrow Y_{0}$
and $D_{2}$, $D_{3}$ are the strict transform of the supports of
two irreducible curves $C_{2}$ and $C_{3}$ defined by the zeros
of $f\in\mathbb{C}[u,v]$ and $g\in\mathbb{C}[u,v]$ respectively.

Fourth step: the topology of the support of $\mathcal{D}$.

Once again by the smoothness criterion the support of $\mathcal{D}$
has to be a SNC-divisor, in particular each irreducible component
of the divisors is smooth. Moreover using the previous presentation
and \cite{P}, $X$ admits two actions of finite cyclic group $\mu_{\zeta}$
and $\mu_{\xi}$ which factorize by the $\mathbb{G}_{m}$ action.
In particular $X$ can be viewed as a bi-cyclic cover of $\mathbb{A}^{3}$
of order $\zeta$ and $\xi$, and it admits the following presentation:

\[
X=\mathrm{Spec}\left(\mathbb{C}[x_{1},x_{2},x_{3},x_{4},x_{5}]/\left(\begin{array}{c}
x_{4}^{\zeta}=f(x_{1}x_{2},x_{1}x_{3})/x_{1}\\
x_{5}^{\xi}=g(x_{1}x_{2},x_{1}x_{3})/x_{1}
\end{array}\right)\right).
\]

Thus we obtain the following tower of threefolds where $X//(\mu_{\zeta}\times\mu_{\xi})$
is equivariantly isomorphic to $\mathbb{A}^{3}$ with a linear $\mathbb{G}_{m}$-action
:

\[
\xyC{2pc}\xymatrix{ & X\ar[ld]\ar[rd]\\
X//\mu_{\zeta}\ar[rd] &  & X//\mu_{\xi}\ar[ld]\\
 & X//(\mu_{\zeta}\times\mu_{\xi})\simeq\mathbb{A}^{3}
}
\]

As $X$ is contractible, so are of $X//\mu_{\zeta}$ and $X//\mu_{\xi}$
by \cite[Theorem B]{KrPRa}. These varieties are obtained by two modifications:
first a hyperbolic modification of $f$ (or $g$ respectively) see
Definition \ref{Hyperbolic modification}. By this way we obtain two
varieties $X_{f}=\{f(x_{1}x_{2},x_{1}x_{3})/x_{1}=0\}\subset\mathbb{A}^{3}$
and $X_{g}\{g(x_{1}x_{2},x_{1}x_{3})/x_{1}=0\}\subset\mathbb{A}^{3}$
which are stable for the following hyperbolic $\mathbb{G}_{m}$-action
on $\mathbb{A}^{3}$, $\lambda\cdot(x_{1},x_{2},x_{3})=(\lambda^{-1}x_{1},\lambda x_{2},\lambda x_{3})$. 

The second modification is a cyclic cover of order $\zeta$ (or $\xi$
respectively) of $\mathbb{A}^{3}=\mathrm{Spec}(\mathbb{C}[x_{1},x_{2},x_{3}])$
along the variety $\{f(x_{1}x_{2},x_{1}x_{3})/x_{1}=0\}$ (or $\{g(x_{1}x_{2},x_{1}x_{3})/x_{1}=0\}$
respectively). By \cite[Theorem A]{Ka}, as the varieties $X//\mu_{\zeta}$
and $X//\mu_{\xi}$ are contractible, the sub-varieties $\{f(x_{1}x_{2},x_{1}x_{3})/x_{1}=0\}$
and $\{g(x_{1}x_{2},x_{1}x_{3})/x_{1}=0\}$ of $\mathbb{A}^{3}=\mathrm{Spec}(\mathbb{C}[x_{1},x_{2},x_{3}])$
have to be $\mathbb{Z}_{k}$-acyclic for almost every $k$, that is
the $\mathbb{Z}_{k}$-homology of the point for almost every $k$.
By a classical classification result, this is possible if and only
if the smooth curves defined by $\{f(u,v)=0\}$ and $\{g(u,v)=0\}$
in $\mathbb{A}^{2}=\mathrm{Spec}(\mathbb{C}[u,v])$ are acyclic and
so, are two copies of the affine line $\mathbb{A}^{1}$.

To summarize, every $\mathbb{A}^{3}$ endowed with a hyperbolic $\mathbb{G}_{m}$-action
is characterized by a pair $(Y,\mathcal{D})$ such that $\mathbb{A}^{3}\simeq\mathbb{S}(Y,\mathcal{D})$.
The A-H quotient $Y$ of $\mathbb{A}^{3}$ is also the A-H quotient
of an $\mathbb{A}^{3}$ endowed with a hyperbolic and linear $\mathbb{G}_{m}$-action.
The coefficient of $\mathcal{D}$ is the same as those used in the
presentation of the $\mathbb{A}^{3}$ endowed with a hyperbolic linear
$\mathbb{G}_{m}$-action and the support of $\mathcal{D}$ is the
union of the exceptional divisor, the strict transform of two curves,
$C_{2}$ and $C_{3}$, both isomorphic to a line in $\mathbb{A}^{2}$,
given by polynomials of weights $a_{2}$, $a_{3}$ respectively and
intersecting normally at the origin and in $(n-1)$ other points,
where $n\geq1$.

Now as the algebraic quotient of $X$ is not necessarily $\mathbb{A}^{2}$,
it has been shown that there exists a cyclic group of order $\mu$
such that the algebraic quotient is $\mathbb{A}^{2}/\mu$. The tangent
space of $X$ at the fixed point is an affine three-space with a linear
hyperbolic $\mathbb{G}_{m}$-action given by $\lambda\cdot(x,y,z)=(\lambda^{-a_{1}}x,\lambda^{a_{2}}y,\lambda^{a_{3}}z)$
with $a_{1}$, $a_{2}$ and $a_{3}$ positive integers. By construction
the order of the finite cyclic group $\mu$ can be chosen as to be
$a_{1}$. The general case is obtained as a quotient of the previous
case by the cyclic group. In particular as proved in \cite{P}, the
two divisor $D_{2}$ and $D_{3}$ have to be invariant for the induced
(and possibly trivial) action on the A-H quotient, so $C_{2}$ and
$C_{3}$ are homogeneous for $\mu_{a_{1}}$.
\end{proof}
All possible varieties that can be built, and that verify the previous
Theorem are not necessary $\mathbb{A}^{3}$. There is a dichotomy,
obtained in \cite{KaML}, using the additive group action $\mathbb{G}_{a}$
on them, in particular they prove that the one class of varieties
obtained by Koras and Russell in \cite{KR4} which are indeed $\mathbb{A}^{3}$
are those which are ``obviously'' $\mathbb{A}^{3}$ with a linear
hyperbolic $\mathbb{G}_{m}$-action. This corresponds to the case
where $D_{2}$ and $D_{3}$ are the two coordinate axes in the same
coordinate system in Theorem \ref{thm:A-smooth,-contractible}. The,
remaining varieties are called Koras-Russell threefolds and classified
in three types according to the richness of the $\mathbb{G}_{a}$-actions
on them. The A-H presentation of these varieties has been computed
in \cite{P}. 
\begin{example}
Let $X$ be the $\mathbb{G}_{m}$-variety equivariantly isomorphic
to $\mathbb{S}(\tilde{\mathbb{A}}_{(u,v^{d})}^{2},\mathcal{D})$ where
$\pi:\tilde{\mathbb{A}}_{(u,v^{d})}^{2}\rightarrow\mathbb{A}^{2}$
is the blow-up of $\mathbb{A}^{2}=\mathrm{Spec}(\mathbb{C}[u,v])$
with center at the closed subscheme defined by the ideal $\mathcal{I}=(u,v^{d})$,
the p-s divisor $\mathcal{D}$ is of the form

\[
\mathcal{D}=\left\{ \frac{1}{3}\right\} \otimes D_{1}+\left\{ \frac{-3}{5}\right\} \otimes D_{2}+\left[0,\frac{1}{15}\right]\otimes E,
\]
 where $E$ is the exceptional divisor of the blow-up $\pi$ and $D_{1}$,
$D_{2}$ are the strict transforms of $\{u=0\}$ and $\{v+(u+v^{2})^{3}=0\}$,
then $X$ is a smooth contractible threefold endowed with an hyperbolic
$\mathbb{G}_{m}$-action but it is not $\mathbb{A}^{3}$. In particular
$X=\{x+y^{5}(x^{2}+z^{3})^{3}+t^{5}=0\}\subset\mathbb{A}^{4}=\mathrm{Spec}(\mathbb{C}[x,y,z,t])$.
\end{example}

\end{document}